\documentclass[12pt,reqno]{amsart}
\usepackage[T1]{fontenc}
\usepackage[english]{babel}
\usepackage[utf8]{inputenc}
\usepackage{amssymb}
\usepackage{amsmath}
\usepackage{amsfonts}
\usepackage{amsthm}
\usepackage{csquotes}
\usepackage{enumitem}
\usepackage{float}
\usepackage[a4paper,margin=2.5cm]{geometry}
\usepackage{textgreek}
\usepackage{enumitem}

\usepackage[doi=false,giveninits=true,isbn=false,url=false]{biblatex}

\AtEveryBibitem{\clearfield{pages}}  
\AtEveryBibitem{\clearfield{number}}

\usepackage{hyperref}
\hypersetup{
    colorlinks=true,
    linkcolor=blue,
    filecolor=magenta,      
    urlcolor=cyan,
}  

\newcommand{\sent}{\leftrightarrow}
\newcommand{\norm}[1]{\left\|#1\right\|}
\newcommand{\N}{\mathbb{N}}
\newcommand{\C}{\mathbb{C}}

\newcommand{\Z}{\mathbb{Z}}
\newcommand{\A}{\mathcal{A}}
\newcommand{\G}{\mathcal{G}}
\newcommand{\T}{\mathcal{T}}
\newcommand{\Path}{\mathcal{P}}
\newcommand{\Fock}{\mathcal{F}}
\newcommand{\Up}{\varepsilon^{\uparrow}}
\newcommand{\Down}{\varepsilon^{\downarrow}}

\newcommand{\NC}{\operatorname{NC}}
\newcommand{\homo}{\operatorname{hom}}
\newcommand{\Homo}{\operatorname{Hom}}
\numberwithin{equation}{section}

\title[Khintchine inequalities, trace monoids and Tur\'an-type problems]{Khintchine inequalities, trace monoids and Tur\'an-type problems}
\author[P.\ O.\ Santos \and R. Tripathi and P.\ Youssef]{%
        Patrick Oliveira Santos \and 
        Raghavendra Tripathi \and
        Pierre Youssef
        }

\address{Patrick Oliveira Santos. Division of Science, NYU Abu Dhabi, Abu Dhabi, UAE.}
\email{po2150@nyu.edu}
\address{Raghavendra Tripathi. Division of Science, NYU Abu Dhabi, Abu Dhabi, UAE.}
\email{r.tripathi@nyu.edu}
\address{Pierre Youssef. Division of Science, NYU Abu Dhabi, Abu Dhabi, UAE \& Courant Institute of Mathematical Sciences, New York University, New York, USA.}
\email{yp27@nyu.edu}
\date{\today}

\begin{document}
\newtheorem{theorem}{Theorem}[section]
\newtheorem{question}{Question}[section]
\theoremstyle{plain}

\newtheorem{corollary}[theorem]{Corollary}
\newtheorem{lemma}[theorem]{Lemma}
\newtheorem{conjecture}[theorem]{Conjecture}
\newtheorem{proposition}[theorem]{Proposition}

\theoremstyle{definition}
\newtheorem{example}[theorem]{Example}
\newtheorem{definition}[theorem]{Definition}
\newtheorem{remark}[theorem]{Remark}

\begin{abstract}
   We prove scalar and operator-valued Khintchine inequalities for mixtures of free and tensor-independent semicircle variables, interpolating between classical and free Khintchine-type inequalities. Specifically, we characterize the norm of sums of $G$-independent semicircle variables in terms of the spectral radius of the Cayley graph associated with the trace monoid determined by the graph $G$. Our approach relies on a precise correspondence between closed paths in trace monoids and the norms of such operator sums. This correspondence uncovers connections between non-commutative probability, combinatorial group theory, and extremal graph theory. In particular, we formulate Turán-type extremal problems that govern maximal norm growth under classical commutation constraints, and identify the extremal configurations. We hope that the methods and connections developed here will be useful in the study of non-commutative structures constrained by combinatorial symmetries. 
\end{abstract}
\maketitle

\section{Introduction}
\label{sec:intro}
The classical Khintchine's inequality \cite{Khintehine1923, Paley1930} provides a fundamental comparison between the $L_p$ and $L_2$ norms of sums of independent Rademacher variables. Given a sequence of scalars $(\alpha_i)_{i=1}^L$ and a sequence of independent identically distributed (i.i.d) Rademacher variables $(\epsilon_i)_{i=1}^L$, the inequality asserts that 
\begin{equation}
\label{eqn: ClassicalKhintchine}
    A_p\left(\sum_{i=1}^{L}\vert\alpha_i\vert^2\right)^{\frac12} \leq \norm{\sum_{i=1}^{L}\alpha_i\epsilon_i}_{2p}\leq B_p \left(\sum_{i=1}^{L}\vert\alpha_i\vert^2\right)^{\frac12} \;\;,
\end{equation}
for some numerical constants $A_p, B_p$. The optimal constants were completely determined by Haagerup~\cite{Haagerup1981} who showed, in particular, that $B_p\sim \sqrt{p}$ as $p\to \infty$. 
Khintchine's original work laid the foundation for a long line of research connecting probability, functional analysis, and harmonic analysis \cite{Hoffmann,lust1991non, Lindenstrauss68, Pisier12Grothendieck, Marcus81random}. Khintchine-type inequalities, beyond providing norm comparison, reflect deep geometric properties. For instance, \eqref{eqn: ClassicalKhintchine} implies that the linear subspace $\mathcal{R}_p$ generated by Rademacher sums embeds into $\ell_2$, illustrating that the Rademacher system behaves ``almost orthonormally'' in $L_{2p}$. This embedding plays a central role in the study of types and cotypes of Banach spaces~\cite{Pisier86Probabilistic, Diestel01Operator, Pisier03Non}, and has had major implications in the local theory of Banach spaces \cite{Ledoux11}. 

Over the decades, Khintchine-type inequalities have been extended in multiple directions. One classical extension replaces Rademacher variables with other independent random variables, such as standard Gaussians, or more general sub-Gaussian or log-concave distributions \cite[Section 2.6]{Vershynin2018}, \cite{milman1986asymptotic}. Another direction enriches the scalar setting by considering vector-valued coefficients. A more profound and structurally rich generalization of Khintchine's inequality arises in the non-commutative setting, where the classical scalar-valued theory is extended to the operator algebra setting. In this framework, noncommutativity can be introduced in two fundamentally different ways: either on the coefficients or on the random variables themselves. The first direction, pioneered by Lust-Piquard and Pisier in a series of foundational works \cite{lust1986inegalites, lust1991non},  introduces noncommutativity through the operator-valued coefficients; see also \cite{Oliveira2010, Tropp2011, Rudelson1999}. One considers sums of the form $\sum_{i=1}^L \epsilon_i a_i$ where the $a_i$'s are bounded linear operators on a Hilbert space and estimates the corresponding $p$-Schatten norm. Remarkably,  the inequalities retain the same order of growth as in the scalar case, with constants behaving like $\sqrt{p}$ as $p\to \infty$ (more precisely, one obtains a bound of order $\min(\sqrt{p},\sqrt{L})$. The case $p=\infty$,  corresponding to the operator norm, is of particular interest and has received special attention due to its connections with random matrix theory \cite{Bandeira23Matrix, Tropp15}. 
The second generalization occurs when the randomness itself becomes non-commutative, as in the framework of free probability introduced by Voiculescu \cite{voiculescu1985symmetries}. Here, classical independence is replaced by freeness, a concept that captures the asymptotic behavior of large random matrices. In this setting, Haagerup and Pisier \cite{haagerup1993,parcet2005non} established a free analog of Khintchine's inequality by analyzing sums of the form $\sum_{i=1}^L a_i\otimes s_i$, where the $s_i$'s are free semicircle variables and the $a_i$'s are bounded linear operators on a Hilbert space. 
The behavior in the free setting is fundamentally different from the classical case. Indeed, when the operators $a_i$'s are replaced by scalars $\alpha_i$'s, the sum $\sum_{i=1}^L \alpha_i s_i$ is simply a semicircle variable with variance $\sum_{i=1}^L \alpha_i^2$. The $2p$-th moment of such variable is known to be $C_p (\sum_{i=1}^L \alpha_i^2)^p$, where $C_p=\frac{1}{2p+1}\binom{2p}{p}$ is the $p$-Catalan number. As $C_p^{\frac{1}{2p}}\to 2$ as $p\to\infty$, this shows that the constants in the corresponding Khintchine-type inequalities remain uniformly bounded in $p$, in sharp contrast to the classical case. This uniform behavior persists even in the operator-valued setting, where Haagerup and Pisier \cite{haagerup1993} proved
\begin{equation}
    \label{eqn: HaagerupPisier}
     \norm{\sum_{i=1}^{L}a_i\otimes s_{i}}\leq 2\max\left\{\norm{\sum_{i=1}^{L}a_ia_i^{*}}^{1/2}, \norm{\sum_{i=1}^{L}a_i^{*}a_i}^{1/2}\right\}.
\end{equation}
This is the free analog of the classical Khintchine inequality in the norm case ($p=\infty$). The same inequality holds when the semicircle variables are replaced by the left-regular representation of the free group~\cite{haagerup1993,parcet2005non}. 

While the classical and free Khintchine inequalities are structurally distinct, it is natural to check what intermediate behavior can emerge when classical and free independence coexist. The goal of this paper is to investigate such intermediate regimes, modeled by mixtures of classical and free independence, and study their impact on Khintchine-type inequalities. To formalize this setting, we use the notion of $G$-independence introduced by M{\l}otkowski \cite{mlotkowski2004lambdafree} and further developed in \cite{speicherjanusz2016mixture}. In this setup, a collection of variables is indexed by the vertex set of a graph $G$, where the edges of $G$ specify the classically independent pairs and the non-edges correspond to free (i.e., non-commuting) pairs. More formally, a collection of $G$-independent variables is constructed on a $C^{*}$- probability space $(A, \tau)$ equipped with a faithful trace $\tau$ and norm $\norm{\cdot}$. This notion, which interpolates between classical independence (when $G$ is a complete graph) and freeness (when $G$ is the edgeless graph), has attracted increasing interest in recent years, particularly due to its emergence in random matrix theory. Random matrix models that exhibit asymptotic $G$-independence have been constructed in~\cites{charlesworth2021matrix,charlesworth2024random}, with strong convergence results established in~\cite{magee2308strongly, collins2023spectrum}. 
We consider a collection of $G$-independent semicircle variables and aim at establishing a corresponding Khintchine-type inequality. We note that such a family appears naturally as the strong limit of a tensor GUE model \cite{chen2024new} (see also \cite{magee2308strongly,collins2024strong}) and, in some cases, characterizes the universal law governing central limit theorems for $\epsilon$-independence \cite{lancien2024centrallimittheoremtensor, cébron2024graphontheoreticapproachcentrallimit,cébron2025centrallimittheoremepsilonindependent}. For instance, the distribution of 
\begin{align*}
    T_{G}:=\frac{1}{\sqrt{L}}\sum_{i=1}^{L}s_{i},
\end{align*}
where $(s_{i})_{i=1}^{L}$ are $G$-independent semicircles, coincides with the graphon measure associated to $G$ and appears as such as the limiting distribution in some ``graphon'' central limit theorem (see \cite{cébron2024graphontheoreticapproachcentrallimit} for precise definitions). A Khintchine-type inequality would, in particular, provide information on the support of this limiting measure.  

 However, even in the seemingly simpler scalar case (with all coefficients equal to $1/\sqrt{L}$), the behavior of these mixtures remains highly nontrivial. To understand the influence of the graph structure on moment growth, we formulate the following extremal question: 
\begin{question}\label{q:turan}
    Given a fixed integer $\omega$, what is the maximum/minimum value of $\norm{T_G}$ among all graphs $G$ on $L$ vertices with clique number $\omega$?
\end{question}
This formulation is reminiscent of Turán-type extremal problems in graph theory \cite{mantel1907problem28,turan1941extremal} (see \cite{Zhao2023} for an introduction to the topic), where one seeks to optimize a graph parameter under constraints on the size of cliques. In our setting, bounding the clique number controls the extent of classical independence (or commutation), and the question becomes: how ``much'' classical independence can be allowed before the behavior deviates significantly from the free case? Moreover, a natural question is to identify which graphs achieve extremal values and, in particular, whether certain well-structured families are candidates for extremizers.

Our first result provides a sharp Khintchine-type estimate for (scalar) linear combinations of $G$-independent semicircles. Throughout, we assume that our variables are defined on a $C^*$-probability space $(\A, \tau)$ equipped with a faithful trace $\tau$. For an element $u\in \A$, we use $\norm{u}_{2p}^{2p}$ to denote $\tau\left({(uu^*)}^{p}\right)$.
\begin{theorem}[Scalar Khintchine inequality for mixtures]
\label{thm: ScalarCoefficientSharpUpperBound}
 Let $G$ be a graph on $[L]$ and let $(s_{i})_{i=1}^{L}$ be a collection of $G$-independent semicircles. For any collection $(\alpha_{i})_{i=1}^{L}$ of complex numbers and any $p\ge 1$, we have
    \begin{align*}
        \norm{\sum_{i=1}^{L}\alpha_is_{i}}_{2p}\leq C_p^{\frac{1}{2p}}\min\left\{\sum_{i=1}^{L}|\alpha_i|^2 |c^*(i)|,\;\; p\sum_{i=1}^{L}|\alpha_i|^2\right\}^{1/2},
    \end{align*}
    where $|c^*(i)|$ is the size of the largest clique in $G$ containing $i$. In particular, 
    \begin{align*}
        \norm{\sum_{i=1}^{L}\alpha_is_{i}}_{2p}\leq C_p^{\frac{1}{2p}}\sqrt{\omega(G)\wedge p}\,  \big(\sum_{i=1}^{L}|\alpha_i|^2\big)^{1/2},
    \end{align*}
    where $\omega(G)$ is the clique number of $G$.
\end{theorem}
This interpolates between the free and classical Khintchine's inequality~\eqref{eqn: ClassicalKhintchine} (with semicircles). Indeed, if $G$ is an edgeless graph (free setting), we have $\omega(G)=1$, while taking $G$ to be a complete graph, we obtain 
\[\norm{\sum_{i=1}^{L}\alpha_is_{i}}_{2p}\leq C_p^{\frac{1}{2p}}\sqrt{L\wedge p}\left(\sum_{i=1}^{L}|\alpha_i|^2\right)^{1/2}.\]
The classical Khintchine's inequality is often stated for $L=\infty$ in which case we obtain~\eqref{eqn: ClassicalKhintchine} with the right dependence on $p$. It is well-known that $\norm{u}_{2p}\to \norm{u}$, when $p$ goes to infinity, for any $u\in (\A, \tau)$ (see~\cite[Proposition 3.17]{nica2006lectures}). Using this,  Theorem~\ref{thm: ScalarCoefficientSharpUpperBound} readily provides an answer to Question~\ref{q:turan} and uncovers the extremizer. 

\begin{corollary}[Graph maximizer of scalar Khintchine]\label{cor: turan}
    Given a fixed integer $\omega$ which divides $L$, and a graph $G$ on $L$ vertices with clique number $\omega$, we have 
    $$
    \norm{T_{G}}\leq \norm{T_{K}},
    $$
    where $K$ denotes the Turán graph, that is, the complete $\omega$-partite graph with equal parts of size $L/\omega$. Moreover, we have
    \begin{align*}
        \norm{T_K}=2\sqrt{\omega}.
    \end{align*}
\end{corollary}
The computation of $\norm{T_K}$ follows by the equality in distribution
\begin{align*}
    T_K\overset{d}{=}\frac{1}{\sqrt{\omega}}\left(s_1+\cdots+s_\omega\right),
\end{align*}
where $s_1,\ldots,s_\omega$ are classical independent semicircles.

 \begin{proposition}[Graph minimizer of scalar Khintchine]
\label{prop: Turan_minimizer}
    Given a fixed integer $\omega$, and a graph $G$ on $L$ vertices with clique number $\omega$, we have
    \begin{align*}
        \norm{T_G} \geq \norm{T_{N}},
    \end{align*}
    where $N$ is the disjoint union of a complete graph $K_\omega$ on $\omega$ vertices and the edgeless graph $N_{L-\omega}$ on $L-\omega$ vertices. Moreover, we have
    \begin{align*}
        \sqrt{\frac{\omega^2+L-\omega}{L}}\le \norm{T_N} \le 2\sqrt{\frac{\omega^2+L-\omega}{L}}.
    \end{align*}
\end{proposition}
The fact that $T_{N}$ minimizes the norm follows from the monotonicity of the norm with respect to the addition of edges (see Remark \ref{remark: monotonicity}). The upper bound on $\norm{T_N}$ follows from Theorem \ref{thm: ScalarCoefficientSharpUpperBound}, whereas the lower bound follows from \cite[Theorem 3.4]{collins2025operatorvaluedkhintchineinequalityepsilonfree}.

While Corollary~\ref{cor: turan} follows as a consequence of our proof of Theorem~\ref{thm: ScalarCoefficientSharpUpperBound}, it would be interesting to obtain it directly via Turán-type arguments. Our proof of Theorem \ref{thm: ScalarCoefficientSharpUpperBound}, however, uncovers a striking connection with the enumeration of closed paths in the Cayley graph of the trace monoid associated with the graph $G$ (see Section~\ref{sec:tracemonoid}). To the best of our knowledge, this link with trace monoids is new. In particular, in the context of Corollary~\ref{cor: turan}, it leads to a characterization of $\norm{T_G}$ via the spectral radius of the corresponding Cayley graph.

\begin{theorem}[Norm characterization via trace monoid]
\label{thm: NormEqualsNumberOfPaths}
Let $G$ be a graph on $L$ vertices and let $s_1, \ldots, s_L$ be $G$-independent semicircles. Let $A$ denote the adjacency matrix of the Cayley graph of the trace monoid associated with $G$ (see Section~\ref{sec:tracemonoid}).
Then, 
\[\norm{\sum_{i=1}^{L}s_i} = \norm{A}.\]
\end{theorem}

We now turn to the more general setting where the scalar coefficients 
are replaced by operators. In this operator-valued case, we consider sums of the form $\sum_{i=1}^{L}a_i\otimes s_{i}$ where the $s_i$'s are $G$-independent semicircle variables and the $a_i$'s are bounded linear operators on some Hilbert space. This setup extends the free operator-valued Khintchine inequality of Haagerup and Pisier \eqref{eqn: HaagerupPisier} (when $G$ is edgeless). 

\begin{theorem}[Operator-valued Khintchine inequality for mixtures]\label{thm: OperatorvaluedCoefficients}
    Let $G$ be a graph on $[L]$ with clique number $\omega(G)$ and let $(s_{i})_{i=1}^{L}$ be a collection of $G$-independent semicircles. Let $\mathcal{H}$ be a Hilbert space and let $(a_i)_{i=1}^{L}\subseteq B(\mathcal{H})$. Then,
    \begin{align*}
        \norm{\sum_{i=1}^{L} a_i\otimes s_i}_{B(\mathcal{H})\otimes \A} \le 2 \sqrt{\omega(G)} \max\left\{\norm{\sum_{i=1}^{L}a_ia_i^*}_{B(\mathcal{H})}^{1/2}, \norm{\sum_{i=1}^{L}a_i^*a_i}_{B(\mathcal{H})}^{1/2}\right\}\;.
    \end{align*}
\end{theorem}
Here, $B(\mathcal{H})\otimes \A$ denotes the minimal tensor product \cite[Chapter 2]{Pisier2003}. 
Theorem~\ref{thm: OperatorvaluedCoefficients} strengthens a recent result by Collins and Miyagawa \cite[Theorem 1.1]{collins2025operatorvaluedkhintchineinequalityepsilonfree}, who established the following bound 
\begin{equation}
    \label{eqn:AkihiroUpperBound}
     \norm{\sum_{i=1}^{L}a_i\otimes s_{i}}_{B(\mathcal{H})\otimes \A}\leq 2\sqrt{\lambda_1+1}\max\left\{\norm{\sum_{i=1}^{L}a_ia_i^{*}}_{B(\mathcal{H})}^{1/2}, \norm{\sum_{i=1}^{L}a_i^{*}a_i}_{B(\mathcal{H})}^{1/2}\right\}\;,
\end{equation}
where $\lambda_1$ is the largest eigenvalue of the adjacency matrix of the graph $G$. It is a well-known result in graph theory that $\omega(G)\leq \lambda_1(G)+1$. It is worth noting that our proof of Theorem~\ref{thm: OperatorvaluedCoefficients} readily adapts to the setting where the semicircle random variables are replaced with the left-regular representations of the right-angled Artin groups (see Remark~\ref{rem: leftrep}).

The paper is organized as follows. In Section~\ref{sec: prel}, we introduce the necessary definitions and tools. Section~\ref{section: bijection} establishes the formal correspondence with closed paths in trace monoids and, as a consequence, provides a proof of Theorem~\ref{thm: NormEqualsNumberOfPaths}. Section~\ref{section: Cayley graph} is devoted to the proof of Theorem~\ref{thm: ScalarCoefficientSharpUpperBound}, while Section~\ref{section: extensions} addresses the operator-valued setting and proves Theorem~\ref{thm: OperatorvaluedCoefficients}. Finally, in Section~\ref{section: further discussions}, we discuss alternative approaches and open questions.

\section{Preliminaries}\label{sec: prel}
Let $(\A,\tau)$ be a $C^*$-probability space such that $\tau$ is faithful; see \cite[Chapter 3]{nica2006lectures} for definitions. 

\subsection{Classical, Free, and \texorpdfstring{$G$}{G}-independence}
\label{subsec: Definition of independence}
A family $\{\A_i\}_{i\in \N}$ of unital subalgebras of $\A$ is called \emph{independent} if:
\begin{itemize}
    \item The subalgebras commute: $ab=ba$ whenever $a\in \A_i$ and $b\in A_j$ with $i\neq j$; 
    \item For any $m\in \N$, and $a_1\in \A_{i_1}, \ldots, a_m\in \A_{i_m}$ with $\tau(a_1)=\cdots=\tau(a_m)=0$ and all $i_k$ distinct, one has $\tau(a_1\ldots a_m)=0$.
\end{itemize}
Voiculescu introduced freeness (or free independence) as a non-commutative analog of classical independence \cite{voiculescu1985symmetries}. The family $\{\A_i\}_{i\in \N}$ is called \emph{free} if:
\begin{itemize}
    \item For any sequence of elements $a_1\in \A_{i_1}, \ldots, a_m\in \A_{i_m}$ with $\tau(a_1)=\cdots=\tau(a_m)=0$ and $i_1\neq i_2\neq \cdots\neq i_m$, one has $\tau(a_1\ldots a_m)=0$.
\end{itemize}
Unlike classical independence, freeness is associated with the free product rather than the tensor product of algebras. The notion of $G$-independence interpolates between classical and free independence. Formally, let $G=([L], E(G))$ be a simple, loopless graph. Subalgebras $\A_1,\ldots,\A_L \subset \A$ are said to be \emph{$G$-independent} if
\begin{itemize}
    \item For every $(i,j) \in E(G)$, $\A_i$ and $\A_j$ are classically independent (in particular, they commute);
    \item For every $k \ge 1$ and $i \in [L]^k$, if for every repeated index $i_{j_1}=i_{j_2}$, there exists an intermediate index $j_3$ such that $j_1<j_3<j_2$ and $(i_{j_1},i_{j_3}) \notin E(G)$, then for any centered random variables $a_j \in \A_{i_j}$, $j \in [k]$, we have
    \begin{align*}
        \tau(a_1\cdots a_k)=0.
    \end{align*}
\end{itemize}

\subsection{Partitions and moment-cumulant formula} 
Given an integer $p\in \N$, a \emph{pair partition} $\pi=\{V_1, \ldots ,V_{2p}\}$ of $[2p]$ is a collection of disjoint subsets $V_1,\ldots, V_{2p}$, called \emph{blocks}, satisfying
\begin{itemize}
    \item $|V_i|=2$ for all $i=1, \ldots, p$;
    \item $V_1 \cup \cdots \cup V_{p}=[2p]$.
\end{itemize}
We denote the set of all pair partitions by $P_2(2p)$.
We say that two block $V_1,V_2\in \pi$ \emph{cross} if there exists $u_1<u_2<v_1<v_2$ such that $V_1=\{u_1,v_1\}$ and $V_2=\{u_2,v_2\}$. For $\pi \in P_2(2p)$, let $F_\pi$ be the \textit{intersection graph} of $\pi$: a graph over the blocks of $\pi$, with edges between any two crossing blocks.

Given $\pi\in P_2(2p)$ and $i\in [L]^{2p}$, we denote $\pi \preceq i$ if $\{u,v\}\in \pi$ implies that $i_u=i_v$. Given $G=([L],E(G))$ and $i\in [L]^{2p}$, we define the set of $(G,i)$-noncrossing partitions by 
\begin{align}\label{eq: def-non-crossing eps partitions}
\NC_2(G,i):=\big\{\pi \in P_2(2p):\; &\pi \preceq i \text{ and }\, \forall\, u_1<u_2<v_1<v_2:\nonumber\\
& \{u_1,v_1\},\{u_2,v_2\}\in \pi\, \Rightarrow (i_{u_1},i_{u_2}) \in E(G)\big\} 
\end{align}
The following lemma provides the moment-cumulant formula for $ G$-independent semicircle random variables.
\begin{lemma}[\cite{speicherjanusz2016mixture}]\label{lemma: speicher}
    Let $G=([L],E(G))$ be a simple graph and $i \in [L]^{2p}$. Let $s_1,\ldots,s_L$ be $G$-independent semicirlce random variables. Then,
    \begin{align*}
        \tau(s_{i_1}\cdots s_{i_{2p}})=|\NC_2(G,i)|.
    \end{align*}
\end{lemma}

\subsection{Graph homomorphisms and moments}
For two graphs $F, G$, let $\Homo(F ,G)$ denote the set of graph homomorphisms from $F$ to $G$, and let $\homo(F, G)=|\Homo(F, G)|$ denote its cardinality. The following lemma, expressing moments in terms of graph homomorphisms, will be used repeatedly throughout the paper.
\begin{lemma}\label{lem: moments of convolution in terms of homomorphisms}
    Let $G$ be a graph on $[L]$, and $s_1,\ldots,s_L$ be $G$-independent semicircle random variables. Then
    \begin{align*}
        \tau \left[\left(\sum_{i=1}^Ls_i\right)^{2p}\right]=\sum_{\pi\in P_2(2p)} \hom(F_\pi,G).
    \end{align*}
    Moreover, let $\alpha_1,\ldots,\alpha_L\in \C$ and define $S_{G}=\sum_{i=1}^L\alpha_i s_i$. Then
    \begin{align*}
        \tau [(S_{G}S_{G}^*)^{p}]=\sum_{\pi\in P_2(2p)}\sum_{\substack{i\in [L]^{2p}\\ i\in \Homo(F_\pi,G)}}\alpha_{i_1}\overline{\alpha}_{i_2}\cdots \alpha_{i_{2p-1}}\overline{\alpha}_{i_{2p}},
    \end{align*}
    where $i\in \Homo(F_\pi,G)$ means that $\pi \preceq i$ and $i:V(F_\pi)\to V(G)\in \Homo(F_\pi,G)$.
\end{lemma}
\begin{remark}[Monotonicity of $\norm{T_G}$]\label{remark: monotonicity}
    Note that Lemma \ref{lem: moments of convolution in terms of homomorphisms} implies that $\norm{T_G}$ is monotone with respect to the addition of edges in $G$. That is, if $G'$ is obtained from $G$ by adding one edge $E(G')=E(G)\sqcup \{e\}$, then 
    \begin{align*}
        \norm{T_G}\le \norm{T_{G'}}.
    \end{align*}
    This follows directly as $\hom(F,G)\le \hom(F,G')$, for any graph $F$ and $\norm{T_G}_{2p}\to \norm{T_G}$ as $p$ goes to infinity.
\end{remark}
\begin{proof}
  It suffices to prove the statement for $S_{G}$. We first write
    \begin{align*}
        \tau [(S_{G}S_{G}^*)^{p}]=\sum_{i\in [L]^{2p}}\alpha_{i_1}\overline{\alpha}_{i_2}\cdots \alpha_{i_{2p-1}}\overline{\alpha}_{i_{2p}}\tau(s_{i_1}\cdots s_{i_{2p}}).
    \end{align*}
    By Lemma \ref{lemma: speicher} and the fact that the $s_i$'s are semicircles, we have
    \begin{align*}
        \tau [(S_{G}S_{G}^*)^{p}]=\sum_{i\in [L]^{2p}}\sum_{\pi\in \NC_2(G,i)}\alpha_{i_1}\overline{\alpha}_{i_2}\cdots \alpha_{i_{2p-1}}\overline{\alpha}_{i_{2p}}.
    \end{align*}
    It remains to observe that a partition $\pi\in P_2(2p)$ belongs to $\NC_2(G,i)$ if and only if $i\in \Homo(F_\pi,G)$, completing the proof.
\end{proof}

\subsection{Trace monoid}
\label{sec:tracemonoid}
The second key ingredient in the proof of Theorem~\ref{thm: ScalarCoefficientSharpUpperBound} is the notion of trace monoids (see \cite{Cartier1969} for an introduction from a combinatorial perspective). Given a graph $G$, we define the trace monoid $\T(G)$ by the  presentation
\begin{align*}
    \T(G)=\langle v\in V(G): uv=vu, \forall (u,v)\in E(G)\rangle.
\end{align*}
That is, we view the vertices $v\in V(G)$ as letters in an alphabet, and impose commutation relations $uv=vu$ whenever $u$ and $v$ are adjacent in $G$. We denote by $e\in \T(G)$ the empty word, satisfying $ew=we=w$ for all $w\in \T(G)$. For a word $w=v_1\cdots v_n\in \T(G)$, where $v_1,\ldots, v_n\in V(G)$, we define its length as $|w|=n$. 
A word $c=u_1\cdots u_m$ in the letters $u_1,\ldots,u_m$ is a clique in the trace monoid if its letters form a clique in $G$, namely, $(u_\ell,u_{\ell'})\in E(G)$ for all $\ell\ne \ell'$.
Two cliques $c=v_{t_1}\cdots v_{t_m}$ and $c'=v_{r_1}\cdots v_{r_n}$ are said to be in \emph{normal position}, denoted $c \to c'$, if for every letter $v_{r_j}$ in $c'$, there exists a letter $v_{t_i}$ in $c$ such that $(v_{t_i},v_{r_j})\notin E(G)$; that is, no letter in $c'$ commutes with all letters in $c$. 
The following structural decomposition result will play a central role in our analysis:
\begin{lemma}[Theorem 1.2 in \cite{Cartier1969}]\label{lemma: Cartier-Foata decomposition}
    Let $w\in \T(G)$. Then, there exist unique integer $h\ge 1$ and normal sequence of cliques $c_1 \to c_2, \ldots ,c_{h-1}\to c_h$ in $\T(G)$ such that $w=c_1\cdots c_h$. We refer to $(c_1,\ldots,c_h)$ as the Cartier-Foata decomposition of $w$.
\end{lemma}

We say that two words $w_1,w_2\in \T(G)$ are \emph{adjacent}, denoted $w_1\sent w_2$, if there exists a letter $v\in V(G)$ such that $w_1=vw_2$ or $w_2=vw_1$. The Cayley graph $\operatorname{Cay}(\T(G))$ of the trace monoid $\T(G)$ is a graph on the vertex set $\T(G)$ with edges between the adjacent words. 
We now define the set of closed paths of length $2p$ on $\operatorname{Cay}(\T(G))$ by 
$$
\Path_{2p}(G)=\{w\in [\T(G)]^{2p}: e \sent w_1 \sent \cdots \sent w_{2p}= e\}.
$$

\section{Norm characterization via trace monoids}
\label{section: bijection}

The goal of this section is to establish a correspondence between the moments of sums of $G$-independent semicircles and closed paths on the associated trace monoid. This correspondence will yield an immediate proof of Theorem~\ref{thm: NormEqualsNumberOfPaths}. 
In light of Lemma~\ref{lem: moments of convolution in terms of homomorphisms}, it suffices to construct a bijection between $\Path_{2p}(G)$ and
\begin{align*}
    &\Pi_{2p}(G):=\{(\pi, \phi): \pi\in P_2(2p), \phi\in \Homo(F_\pi,G)\}.
\end{align*}
Informally, the bijection $\Phi_p$ is constructed as follows. Given $(\pi, \phi)\in \Pi_{2p}(G)$, we associate a path $\Phi_p(\pi, \phi)$ by processing each block $U=\{s,r\}\in \pi$, with $s<r$ and $i_s=\phi(U)$. At step $s$, we append the letter $i_s$ to the current word, and at step $r$, we remove it. We interpret $s$ as the starting index of $U$ and $r$ as the corresponding return index. For example, consider two commuting letters $(i_1,i_2)\in E(G)$, and let $\pi=\{V_1,V_2\}$ with $V_1=\{1,3\},V_2=\{2,4\}$, and $\phi(V_1)=i_1,\phi(V_2)=i_2$. Then, the corresponding path is
\begin{align*}
    \Phi_2(\pi,\phi)=(i_1, i_2i_1,i_2,e).
\end{align*}
The removal of the letter $i_1$ at the third step is justified as the letters commute.  

Formally, we define recursively a family of mappings $\Phi_{p}:\, \Pi_{2p}(G) \to [\T(G)]^{2p}$ as follows. 
\begin{enumerate}
    \item\label{definition: phi_1} For $p=1$, we have $P_2(2)=\{\pi\}$, where $\pi$ contains the unique block $V=\{1,2\}$. Given $i\in [L]$, let $\phi_i\in \Homo(F_\pi,G)$ be such that $\phi_i(V)=i$. Then, we define
    \begin{align*}
        \Phi_1(\pi,\phi_i)=(i,e).
    \end{align*}
    \item\label{definition: phi_p+1 given phi_p} Assume $\Phi_p$ is defined. Let $(\pi,\phi)\in \Pi_{2(p+1)}(G)$. Let $r\in [2(p+1)]$ be the smallest index such that there exists $s<r$ with $U=\{s,r\}\in \pi$; that is, $r$ is the first return index of a block in $\pi$. Let $i_s=\phi(U)$, and let $\sigma=\pi\setminus\{U\}$ and $\psi=\phi|_\sigma$. Identify
    \begin{align*}
        P_2([2(p+1)]\setminus\{s,r\})\simeq P_2(2p),
    \end{align*}
    by the unique order-preserving bijection between $[2(p+1)]\setminus\{s,r\}$ and $[2p]$. 
    Let $u=\Phi_p(\sigma,\psi)\in [\T(G)]^{2p}$ and write 
    \begin{align*}
        u=(u_k)_{k\in [2(p+1)]\setminus\{s,r\}}.
    \end{align*}
    Then, define
    \begin{align*}
        (\Phi_{p+1}(\pi,\phi))_k=\begin{cases}
            u_k; &\text{ if } k<s;\\
            i_s u_{s-1}; &\text{ if } k=s;\\
            i_su_{k};&\text{ if } s<k<r;\\
            u_{r-1};&\text{ if }k=r;\\
            u_k;& \text{ if } k>r. 
        \end{cases}
    \end{align*}
\end{enumerate}
We refer to the tuple $(i_s,s,r,U,\sigma,\psi)$ as defined in \eqref{definition: phi_p+1 given phi_p}, as the \emph{first-returning block decomposition} of $(\pi,\phi)$.
We begin by establishing that the mapping $\Phi_p$ increases the length of the word at the starting index of a block and decreases it at the corresponding return index. 
\begin{lemma}\label{lemma: phi and addition of letters}
    Let $\Phi_p:\Pi_{2p}(G)\to [\T(G)]^{2p}$ be the mapping defined recursively in \eqref{definition: phi_1} and \eqref{definition: phi_p+1 given phi_p}. Let $(\pi,\phi)\in \Pi_{2p}(G)$ and $w=\Phi_p(\pi,\phi)$. Then, for any block $\{s,r\}\in \pi$ with $s<r$, we have
    \begin{align*}
        |w_s|=|w_{s-1}|+1,\quad \text{and}\;\;\; |w_r|=|w_{r-1}|-1.
    \end{align*}
\end{lemma}
\begin{proof}
    The proof proceeds by induction on $p$. For $p=1$, the claims follows directly from~\eqref{definition: phi_1}. Now assume the statement holds for some $p\geq 1$, and consider $(\pi,\phi)\in \Pi_{2(p+1)}(G)$. Let $(i_{s^*},s^*,r^*,U,\sigma,\psi)$ be the first-returning block decomposition of $(\pi,\phi)$. By \eqref{definition: phi_p+1 given phi_p}, we have
    \begin{align*}
        |w_{s^*}|=|i_{s^*}u_{s^*-1}|=|u_{s^*-1}|+1=|w_{s^*-1}|+1.
    \end{align*}
    Similarly, we have $|w_{r^*}|=|w_{r^*-1}|-1$. For all other blocks $V\in \pi$ with $V\ne U$, we have $V\in \sigma$, and the claim follows from the induction hypothesis applied to $u$.
\end{proof}

The next three lemmas establish that $\Phi_p$ is in fact a bijection between $\Pi_{2p}(G)$ and $\Path_{2p}(G)$.
\begin{lemma}
\label{lemma: range of phi_p }
    Let $\Phi_p:\Pi_{2p}(G)\to [\T(G)]^{2p}$ be the mapping defined in \eqref{definition: phi_1} and \eqref{definition: phi_p+1 given phi_p}. Then, for any $(\pi,\phi)\in \Pi_{2p}(G)$, we have $\Phi_p(\pi,\phi)\in \Path_{2p}(G)$.
\end{lemma}
\begin{proof}
   The proof proceeds by induction on $p$. For $p=1$, the statement follows directly from the definition in \eqref{definition: phi_1}. Assume the lemma holds for some $p\geq 1$, and consider $(\pi,\phi)\in \Pi_{2(p+1)}(G)$ with its first-returning block decomposition $(i_s,s,r,U,\sigma,\psi)$. Let $w=\Phi_{p+1}(\pi,\phi)$. By the induction hypothesis, we have that $u=\Phi_p(\sigma,\psi)\in \Path_{2p}(G)$, so we have $e\sent u_1\sent \cdots \sent u_{2p}=e$. From the construction in \eqref{definition: phi_p+1 given phi_p}, it follows that
    \begin{align*}
        &e \sent w_1 \sent \cdots \sent w_{s-1};\\
        &w_{r}\sent \cdots \sent w_{2p}=e.
    \end{align*}
    Furthermore, since $w_{s}=i_sw_{s-1}$ and $w_{r-1}=i_sw_r$, we also have $w_{s-1}\sent w_s$ and $w_{r-1} \sent w_r$. It remains to show that $ w_s\sent\cdots\ \sent w_{r-1}$. Since $U=\{s,r\}\in \pi$ is the first returning block of $\pi$, for every $s<q<r$, there exists $t_q>r$ such that $U_q=\{q,t_q\}\in \pi$. Hence $U$ and $U_q$ cross, and by the definition of $\phi$, we have $(i_q,i_s)\in E(G)$, where $i_q=\phi(U_q)$. Each such index $q\in [r-1]\setminus\{s\}$ corresponds to the opening of a block in $\sigma$. By Lemma \ref{lemma: phi and addition of letters} and the assumption that $u\in \Path_{2p}(G)$, it follows that 
    \begin{align*}
        u_q=i_{q}i_{q-1}\cdots i_{s+1}i_{s-1}\cdots i_1.
    \end{align*}
    Thus, for $q=s+1,\dots, r-1$, we have
    \begin{align*}
        w_q=i_s u_q=i_s i_q\cdots i_{s+1}u_{s-1}=i_{q}\cdots i_{s+1}w_s. 
    \end{align*}
    This shows that $w_s\sent \cdots \sent w_{r-1}$, and hence $w\in\Path_{2(p+1)}(G)$.
\end{proof}
\begin{lemma}[$\Phi_p$ is injective]
     Let $\Phi_p:\Pi_{2p}(G)\to \Path_{2p}(G)$ be the mapping defined by \eqref{definition: phi_1} and \eqref{definition: phi_p+1 given phi_p}. Then, $\Phi_p$ is injective.
\end{lemma}
\begin{proof}
   We proceed by induction on $p$. The case $p=1$ follows directly from the definition \eqref{definition: phi_1}. Assume now that the statement holds for some $p\geq 1$. Let $(\pi_1,\phi_1), (\pi_2,\phi_2)\in \Pi_{2(p+1)}(G)$ be such that $(\pi_1,\phi_1)\ne (\pi_2,\phi_2)$. For $j\in \{1,2\}$, denote
    \begin{align*}
        &w^{(j)}=\Phi_{p+1}(\pi_j,\phi_j);\\
        &(i_{s_j}^{(j)},s_j,r_j,U_j,\sigma_j,\psi_j) \sim \pi_j;\\
        &u^{(j)}=\Phi_{p}(\sigma_j,\psi_j),
    \end{align*}
    where the second line corresponds to the first-returning block decomposition of $(\pi_j,\phi_j)$. We distinguish several cases based on how $(\pi_1, \phi_1)$ and $(\pi_2, \phi_2)$ differ. \\
    
\noindent\underline{Case 1}: $\pi_1 = \pi_2$ but $\phi_1 \ne \phi_2$. \\
Then $U := U_1 = U_2$, $s := s_1 = s_2$, and $\sigma := \sigma_1 = \sigma_2$. Hence, either $i_s^{(1)} \ne i_s^{(2)}$ or $\psi_1 \ne \psi_2$. 
\begin{itemize}
    \item If $\psi_1 = \psi_2$, then $u^{(1)} = u^{(2)}$, and since $i_s^{(1)} \ne i_s^{(2)}$, the definition \eqref{definition: phi_p+1 given phi_p} implies $w^{(1)} \ne w^{(2)}$ (difference occurs at step $s$).
    \item If $\psi_1 \ne \psi_2$, then by the induction hypothesis, $u^{(1)} \ne u^{(2)}$. Let $t$ be the minimal index $k \in [2(p+1)] \setminus \{s, r\}$ such that $u_k^{(1)} \ne u_k^{(2)}$. 
    \begin{itemize}
        \item If $t < s$ or $t > r$, then by \eqref{definition: phi_p+1 given phi_p}, we have that $w^{(1)}\ne w^{(2)}$.
        \item If $s<t<r$: 
        \begin{itemize}
            \item If $i_s^{(1)} = i_s^{(2)}$, then $w_t^{(1)} \ne w_t^{(2)}$ since $u_t^{(1)} \ne u_t^{(2)}$. 
            \item If $i_s^{(1)} \ne i_s^{(2)}$, then $w_s^{(1)} \ne w_s^{(2)}$.
        \end{itemize}
        In both subcases, we conclude $w^{(1)} \ne w^{(2)}$.
    \end{itemize}
\end{itemize}
         
\noindent\underline{Case 2}: $\pi_1 \ne \pi_2$ and there exists $t_r < r < q_r$ such that $\{t_r, r\} \in \pi_1$ and $\{r, q_r\} \in \pi_2$.\\
In this case, $r$ is a returning index in $\pi_1$, but a starting index in $\pi_2$. Then, by Lemma~\ref{lemma: phi and addition of letters}, 
        \begin{align*}
            &|w_r^{(1)}|<|w_{r-1}^{(1)}|,\quad |w_r^{(2)}|>|w_{r-1}^{(2)}|,
        \end{align*}
which implies $w^{(1)} \ne w^{(2)}$. By symmetry, the same argument applies if $r$ is a starting index in $\pi_1$ but a returning index in $\pi_2$. \\

In the remaining cases, assume that $\pi_1 \ne \pi_2$ but the set of opening indices in $\pi_1$ and $\pi_2$ are identical. In particular, $r := r_1 = r_2$.\\
\noindent\underline{Case 3}: $U_1 = U_2$, $s := s_1 = s_2$, and $i_s^{(1)} = i_s^{(2)}$. \\
Necessarily, we have $\sigma_1\ne\sigma_2$. Then, by the induction hypothesis, $u^{(1)} \ne u^{(2)}$, and by definition \eqref{definition: phi_p+1 given phi_p}, we conclude $w^{(1)} \ne w^{(2)}$.

\noindent\underline{Case 4}: $U_1 = U_2$, $s := s_1 = s_2$, but $i_s^{(1)} \ne i_s^{(2)}$. \\
Then either $w_{s-1}^{(1)} \ne w_{s-1}^{(2)}$, or if $w_{s-1}^{(1)} = w_{s-1}^{(2)}$, we have $w_s^{(1)} \ne w_s^{(2)}$ since $i_s^{(1)} \ne i_s^{(2)}$. In both subcases, $w^{(1)} \ne w^{(2)}$.

\noindent\underline{Case 5}: $U_1 \ne U_2$. \\
Without loss of generality, suppose $s_1 < s_2$. Let $i_q^{(j)}$, $j\in\{1,2\}$, denote the letter added at step $q$ for $w^{(j)}$, for $1\le q<r$, so that 
 \begin{align*}
            &w_q^{(1)}=i_q^{(1)}\cdots i_1^{(1)}, \quad 
            w_q^{(2)}=i_q^{(2)}\cdots i_1^{(2)}.
        \end{align*}
\begin{itemize}
    \item If $i_{s_1}^{(1)} \ne i_{s_2}^{(2)}$, either $w_{r-1}^{(1)} \ne w_{r-1}^{(2)}$, or if $w_{r-1}^{(1)} = w_{r-1}^{(2)}$, the differing letters removed at step $r$ imply $w_r^{(1)} \ne w_r^{(2)}$.
    \item If $i_{s_1}^{(1)} = i_{s_2}^{(2)}$, note that $i_{s_2}^{(1)} \ne i_{s_1}^{(1)}$ since $i_{s_1}^{(1)}$ must commute with all letters $i_q^{(1)}$ for $s_1<q<r$. Again, either $w_{s_2-1}^{(1)} \ne w_{s_2-1}^{(2)}$, or if $w_{s_2-1}^{(1)} = w_{s_2-1}^{(2)}$, the added letters differ at step $s_2$, hence $w_{s_2}^{(1)} \ne w_{s_2}^{(2)}$.
\end{itemize}
Since all possible cases lead to $w^{(1)} \ne w^{(2)}$, we conclude that $\Phi_{p+1}$ is injective and finish the proof. 
\end{proof}
\begin{lemma}[$\Phi_p$ is surjective]
    Let $\Phi_p:\Pi_{2p}(G)\to \Path_{2p}(G)$ be the mapping defined by \eqref{definition: phi_1} and \eqref{definition: phi_p+1 given phi_p}. Then, $\Phi_p$ is surjective.
\end{lemma}
\begin{proof}
    We proceed by induction on $p$. The case $p=1$ is straightforward from the definition. Now suppose $w\in \Path_{2(p+1)}(G)$, and let $r\in [2(p+1)]$ be the smallest index such that $|w_r|<|w_{r-1}|$, that is, $r$ is the first return. Let $i_r$ be the letter such that $w_{r-1}=i_rw_r$. Since $w\in \Path_{2(p+1)}(G)$, the letter $i_r$ must have been added at a previous step. Let $s<r$ be the largest such index such that $i_s=i_r$. Then we have 
    \begin{align*}
        w_{s}=i_sw_{s-1}=i_rw_{s-1},
    \end{align*}
Now, for each $s<q<r$, let $i_q\ne i_r$ be the letter added at step $q$, so that
    \begin{align*}
        w_q=i_qw_{q-1}=i_q\cdots i_{s+1}i_sw_{s-1}.
    \end{align*}
    Since $i_q\ne i_r=i_s$ and $w_r$ results from removing $i_r$, it must be that all letters $i_q$, with $s<q<r$, commute with $i_r$. Therefore, 
    \begin{align}\label{equation: commuting ways w}
        w_q=i_s i_{q}\cdots i_{s+1}w_{s-1};\quad q=s+1,\ldots, r-1.
    \end{align}
    Define a path $u\in [\T(G)]^{[2(p+1)]\setminus\{s,r\}}$ by
    \begin{align*}
        u_k=\begin{cases}
            w_k;&\text{ if } k<s;\\
            i_{q}\cdots i_{s+1}w_{s-1};&\text{ if }s<k<r;\\
            w_k;&\text{ if } k>r.
        \end{cases}
    \end{align*}
    By construction, and since $w\in \Path_{2(p+1)}(G)$, it follows that $u\in \Path_{2p}(G)$, using the order-preserving bijection between $[2(p+1)]\setminus\{s,r\}$ and $[2p]$. From \eqref{equation: commuting ways w}, we also obtain
    \begin{align}\label{equation: w and u}
        w_k=\begin{cases}
            u_k; &\text{ if } k<s;\\
            i_s u_{s-1}; &\text{ if } k=s;\\
            i_su_{k};&\text{ if } s<k<r;\\
            u_{r-1};&\text{ if }k=r;\\
            u_k;& \text{ if } k>r. 
        \end{cases}
    \end{align}
    By the induction hypothesis, there exists $(\sigma,\psi)\in \Pi_{2p}(G)$ such that $u=\Phi_p(\sigma,\psi)$ with $\sigma\in P_2([2(p+1)]\setminus\{s,r\})$. Define the extension $(\pi,\phi)$ of $(\sigma,\psi)$ by 
    \begin{align*}
        &\pi:=\sigma \cup \{\{s,r\}\},\quad \phi(\{s,r\}):=i_r.
    \end{align*}
    Since $(i_q,i_r)\in E(G)$ for all $s<q<r$, we have that $\phi\in \Homo(F_\pi,G)$. Using \eqref{equation: w and u}, we conclude that $w=\Phi_{p+1}(\pi,\phi)$. Hence, $\Phi_{p+1}$ is surjective, completing the proof.
\end{proof}

We conclude this section with the proof of Theorem~\ref{thm: NormEqualsNumberOfPaths}.
\begin{proof}[Proof of Theorem~\ref{thm: NormEqualsNumberOfPaths}]
Let $A$ be the adjacency matrix of the Cayley graph $\operatorname{Cay}(\T(G))$ of the trace monoid $\T(G)$. Note that 
\[|\Path_{2p}(G)|=\langle A^{2p}e,e\rangle\;.\]
On the other hand, Lemma~\ref{lem: moments of convolution in terms of homomorphisms} and the bijection between $\Pi_{2p}(G)$ and $\Path_{2p}(G)$ yield
\[ |\Path_{2p}(G)|=\norm{T}_{2p}^{2p} \;.\]
The desired conclusion now follows by observing that
\[\norm{T_{G}} = \lim_{p\to \infty} \left(\norm{T}_{2p}^{2p}\right)^{\frac{1}{2p}} = \lim_{p\to \infty}\langle A^{2p}e,e\rangle^{\frac{1}{2p}} = \norm{A}.\]
The first equality follows from Proposition \cite[Proposition 3.17]{nica2006lectures}, whereas the last follows from classical results on locally finite graphs; see \cite{Mohar1982}, \cite[Theorem 4.4]{Mohar1989}.
\end{proof}


\section{Scalar Khintchine inequality}
\label{section: Cayley graph} 
Given a path $w\in \Path_{2p}(G)$, we associate to it a Dyck path $\varepsilon=\varepsilon(w)\in D_{2p}$ as 
\begin{align*}
    \varepsilon_{k}=\begin{cases}
        +1;&\text{ if } |w_{k}|>|w_{k-1}|;\\
        -1;&\text{ otherwise.}
    \end{cases}
\end{align*}
It is immediate from the definition of $\Path_{2p}(G)$ that $\varepsilon$ is indeed a Dyck path, meaning its partial sums are nonnegative and the total sum equals zero. For each $\varepsilon \in D_{2p}$, we define a subset of paths corresponding to $\varepsilon$ by 
\[\Path_{2p}^{(\varepsilon)}(G)=\{w\in \Path_{2p}(G): \varepsilon(w)=\varepsilon\};\]
and the index sets of ``up'' and ``down'' steps as
\begin{align*}
\Up=\{k\in [2p]: \varepsilon_k=+1\}\quad\text{and}\quad \Down=\{k\in [2p]: \varepsilon_k=-1\}.
\end{align*}
For each $\varepsilon\in D_{2p}$ and every labelling $i\in [L]^{\Up}$ of up steps in $\varepsilon$, we further define $\Path_{2p}^{(\varepsilon)}(G,i)$, the set of all paths in $\Path_{2p}^{\varepsilon}(G)$ that are `compatible' with the labelling $i$ by
\begin{align*}
    \Path_{2p}^{(\varepsilon)}(G,i)=\{w\in \Path_{2p}^{(\varepsilon)}(G): w_k=i_kw_{k-1} \text{ for all }k\in\Up\}.
\end{align*}
That is, we fix the letter $i_k$ that was added to the word at each up-step $k$, so that $w_{k}$  is obtained from $w_{k-1}$ by left multiplication with $i_k$. The main proposition of this section is the following.
\begin{proposition}\label{proposition: bounding paths on Cayley}
   Let $G$ be a graph on the vertex set $[L]$. For each $v\in [L]$, let $c^*(v)$ denote the size of the largest clique in $G$ containing $v$. For any $\varepsilon\in D_{2p}$ and any $i\in [L]^{\Up}$, we have
    \begin{align*}
        |\Path_{2p}^{(\varepsilon)}(G,i)|\le \min\left\{\prod_{k\in \Up}|c^{*}(i_k)|, p!\right\}\;.
    \end{align*}  
In particular, we have 
\begin{align*}
        |\Path_{2p}^{(\varepsilon)}(G,i)|\le \min( \omega(G)^{p}, p!),
    \end{align*}  
 where $\omega(G)$ denotes the clique number of $G$.    
\end{proposition}
\begin{proof}
Since all labels $i_k$ corresponding to up steps $k\in \Up$ are fixed, the only remaining choices in constructing a path in $\Path_{2p}^{(\varepsilon)}(G, i)$ concern the letters to be removed at each down step. Observe that for any such down step at time $t$, we must have $w_{t-1}= j_tw_t$ for some $j_t\in [L]$. Since $w\in \Path_{2p}(G)$, each letter removed at a down step must be one that was previously added during an up step. That is, every down label $j_t$ must be equal to some previous up label, implying that the collection of down labels is a permutation of the up labels. Consequently, we have the simple upper-bound 
    \begin{align*}
        |\Path_{2p}^{(\varepsilon)}(G,i)|\le p!
    \end{align*}
We now establish the sharper bound 
    \[|\Path_{2p}^{(\varepsilon)}(G,i)|\le \prod_{k\in \Up}|c^{*}(i_k)|\;,\] 
    by induction on $p$. The claim holds trivially for $p=1$. 
    Assume the statement holds for some $p\ge 1$, and consider $\varepsilon\in D_{2(p+1)}$. Let $k+1$ be the index of the first down step in $\varepsilon$, i.e.,  
    \begin{align*}
        k+1=\min\{t\in [2(p+1)]: t\in \Down\}.
    \end{align*}
    Let $w_k=i_k\cdots i_1$ be the word obtained at step $k$. By Lemma \ref{lemma: Cartier-Foata decomposition}, there exists a unique integer $h\ge 1$ and a normal sequence $c_1\to c_2\to \cdots \to c_h$ such that $w_k=c_1\cdots c_h$. Write $c_1=i_{t_1}\cdots i_{t_{|c_1|}}$, where the indices $t_1,\ldots, t_{|c_1|}\in [k]$ are determined by the fixed sequence $i_1,\ldots, i_k$. Then, the letter $j\in [L]$ removed at step $k+1$ must satisfy 
    \begin{align*}
        j\in \{i_{t_1},\cdots ,i_{t_{|c_1|}}\}.
    \end{align*}
    For each $r= 1,\ldots, |c_1|$, define the Dyck path $\varepsilon^{(r)}\in D_{2p}$ by removing the up step at position $t_r$ and the down step at position $k+1$, i.e., 
    \begin{align*}
        \varepsilon^{(r)}=(\varepsilon_1,\ldots, \varepsilon_{t_{r-1}},\varepsilon_{t_{r+1}},\ldots, \varepsilon_{k},\varepsilon_{k+2},\ldots, \varepsilon_{2p})\in D_{2p}.
    \end{align*}
    Then 
    \begin{align*}
        |\Path_{2(p+1)}^{(\varepsilon)}(G,i)|=\sum_{r=1}^{|c_1|}|\Path_{2p}^{(\varepsilon^{(r)})}(G,i^{(r)})|,
    \end{align*}
    where $i^{(r)}\in {\varepsilon^{(r)}}^{\uparrow}$ is the restriction of $i$ to the remaining up steps, i.e., 
    \begin{align*}
        i^{(r)}=(i_k)_{k\in \Up\setminus\{t_r\}}. 
    \end{align*}
    Applying the induction hypothesis to each term in the sum yields 
    \begin{align*}
        |\Path_{2(p+1)}^{(\varepsilon)}(G,i)|\le \sum_{r=1}^{|c_1|}\prod_{k\in  {\varepsilon^{(r)}}^{\uparrow}}|c^*(i_k)|=\prod_{k\in \Up}|c^*(i_k)|\left(\sum_{r=1}^{|c_1|}\frac{1}{|c^*(i_{t_r})|}\right).
    \end{align*}
   To complete the inductive step, note that since $c_1$ is a clique containing each $i_{t_r}$, we have $|c_1| \le |c^*(i_{t_r})|$ for all $r$, and therefore it follows that
    \begin{align*}
        \sum_{r=1}^{|c_1|}\frac{1}{|c^*(i_{t_r})|}\le \frac{|c_1|}{\min_{r}|c^*(i_{t_r})|}\le 1.
    \end{align*}
\end{proof}

Theorem \ref{thm: ScalarCoefficientSharpUpperBound} is a direct consequence of Proposition \ref{proposition: bounding paths on Cayley}.
\begin{proof}[Proof of Theorem~\ref{thm: ScalarCoefficientSharpUpperBound}]
By Lemma \ref{lem: moments of convolution in terms of homomorphisms} and the triangle inequality, we immediately obtain that for any complex coefficients $\alpha_1,\ldots,\alpha_L\in \C$,
\begin{align*}
    \norm{\sum_{i=1}^L\alpha_is_i}_{2p} \le \norm{\sum_{i=1}^L |\alpha_i|s_i}_{2p}.
\end{align*}
It thus suffices to consider the case where all coefficients $\alpha_1, \ldots, \alpha_L$ are real and nonnegative. Using Lemma~\ref{lem: moments of convolution in terms of homomorphisms} together with the bijection $\Phi$ from $A_{2p}(G)$ to $\Path_{2p}(G)$ constructed in Section~\ref{section: bijection}, we have
    \begin{align*}
\norm{\sum_{i=1}^{L}\alpha_is_{i}}_{2p}^{2p}=\sum_{\pi\in P_2(2p)}\sum_{i\in \Homo(F_\pi,G)}\alpha_{i_1}\cdots \alpha_{i_{2p}}=\sum_{w\in \Path_{2p}(G)}\alpha_{i_1^{(w)}}\cdots \alpha_{i_{2p}^{(w)}},
    \end{align*}
    where $i^{(w)}$ denotes the label sequence associated with $w$. In the sequel, we drop the superscript $(w)$ for ease of notation. We now decompose each path $w$ according to its associated Dyck path $\varepsilon\in D_{2p}$ and the labels of the up steps. This yields 
    \begin{align*}
\norm{\sum_{i=1}^{L}\alpha_is_{i}}_{2p}^{2p}=\sum_{\varepsilon\in D_{2p}}\sum_{i\in [L]^{\Up}}\sum_{w\in \Path_{2p}^{(\varepsilon)}(G,i)}\alpha_{i_1}\cdots \alpha_{i_{2p}}.
    \end{align*}
   Since each label appears exactly twice, once at the up step and once it is removed at the matching down step, we can rewrite the product as 
   $$ 
   \alpha_{i_1}\cdots \alpha_{i_{2p}}=\prod_{k\in \Up}\alpha_{i_k}^2.$$ 
   Substituting this expression, we get
    \begin{align*}
\norm{\sum_{i=1}^{L}\alpha_is_{i}}_{2p}^{2p}=\sum_{\varepsilon\in D_{2p}}\sum_{i\in [L]^{\Up}}\left(\prod_{k\in \Up}\alpha_{i_k}^2\right)|\Path_{2p}^{(\varepsilon)}(G,i)|.
    \end{align*}
   Now, we apply Proposition \ref{proposition: bounding paths on Cayley}. On the one hand, we get 
    \begin{align*}
\norm{\sum_{i=1}^{L}\alpha_is_{i}}_{2p}^{2p}\le\sum_{\varepsilon\in D_{2p}}\sum_{i\in [L]^{\Up}}\prod_{k\in \Up}\alpha_{i_k}^2|c^*(i_k)|=C_p\left(\sum_{i=1}^L\alpha_i^2|c^*(i)|\right)^p,
    \end{align*}
    where $C_p = |D_{2p}|$ is the $p$-th Catalan number. 
  On the other hand, using the bound $|\Path_{2p}^{(\varepsilon)}(G,i)|\leq p!\leq p^p$, we obtain
  \[\norm{\sum_{i=1}^{L}\alpha_is_{i}}_{2p}^{2p}\leq C_p\;p^p\left(\sum_{i=1}^L\alpha_i^2\right)^p\;.\]
  Combining the last two inequalities, we deduce 
    \begin{align*}
        \norm{\sum_{i=1}^{L}\alpha_is_{i}}_{2p}\leq C_p^{\frac{1}{2p}}\min\left\{\sum_{i=1}^{L}|\alpha_i|^2 |c^*(i)|,\;\; p\sum_{i=1}^{L}|\alpha_i|^2\right\}^{1/2}\;.
    \end{align*}
\end{proof}


\section{Operator-valued Khintchine inequalities}\label{section: extensions}

In this section, we extend the analysis to the operator-valued setting and prove Theorem~\ref{thm: OperatorvaluedCoefficients}. Our argument relies on a Fock space framework defined by the graph structure of $G$ and a careful analysis of associated creation and annihilation operators. While the proof strategy is reminiscent of that in~\cite[Proposition 4.8]{haagerup1993}, we incorporate key new ingredients that are specific to the $G$-independence setting.


We consider the Fock space construction
\begin{align*}
    \Fock(G)= \bigoplus_{w\in \T(G)}\C\cdot x_w,
\end{align*}
where each $\C\cdot x_w$ is a one-dimensional Hilbert space orthogonal to all others.  We define the left creation $l(x_i)$ for $i\in [L]$ via
\begin{align*}
    l(x_i) x_w=x_{i\cdot w},
\end{align*}
and extend it by linearity to $\Fock(G)$. The corresponding annihilation operator $l^*(x_i)$ is defined by
\begin{align*}
    l^*(x_i)x_w=\begin{cases}
        x_{w'}; &\text{ if }w=i\cdot w' \text{ for some }w'\in \T(G),\\
        0;&\text{ otherwise},
    \end{cases}
\end{align*}
and also extended linearly to $\Fock(G)$.
As shown in \cite{bozejko1994completely}, the operators $s_{i}:=l(x_i)+l^*(x_i)$, for $i\in [L]$, are $G$-independent semicircle elements in the space $B(\Fock(G))$ of bounded linear operators on $\Fock(G)$, equipped with the trace
\begin{align*}
    \rho(b)=\langle b(x_e),x_e\rangle,
\end{align*}
for $b\in B(\Fock(G))$, and we recall that $e\in \T(G)$ is the empty word. Moreover, restricted to the algebra $\mathcal{A}$ generated by $(l(x_i)+l^*(x_i))_{i\in [L]}$, the trace is faithful \cite[Theorem 4.3]{bozejko1994completely}.

\begin{proof}[Proof of Theorem \ref{thm: OperatorvaluedCoefficients}]
    Let $(a_i)_{i=1}^{L}\subseteq B(\mathcal{H})$ be a collection of bounded linear operators on some Hilbert space $\mathcal{H}$. 
    Let $$T=\sum_{i=1}^{L} a_i\otimes s_i=\sum_{i=1}^{L} a_i\otimes (l(x_i)+l^*(x_i)).$$ 
    By the triangle inequality, we have
    \begin{align*}
        \norm{T}_{B(\mathcal{H})\otimes \A}\le \norm{\sum_{i=1}^{L}a_i\otimes l(x_i)}_{B(\mathcal{H})\otimes \A}+ \norm{\sum_{i=1}^{L}a_i\otimes l^*(x_i)}_{B(\mathcal{H})\otimes \A}.
    \end{align*}
    For any $C^*$-algebra $(\mathcal{B},\norm{\cdot})$ and elements $c_1,\ldots,c_L,d_1,\ldots,d_L\in \mathcal{B}$, we have the following Cauchy-Schwarz type inequality
    \begin{align}\label{equation: cauchy-schwarz}
        \norm{\sum_{i=1}^L c_id_i} \le \norm{\sum_{i=1}^L c_ic_i^*}^{1/2}\norm{\sum_{i=1}^Ld_i^*d_i}^{1/2}.
    \end{align}
    We apply \eqref{equation: cauchy-schwarz} to $\mathcal{B}=B(\mathcal{H})\otimes \A$, $c_i=1\otimes l(x_i)$ and $d_i=a_i\otimes 1$ where $1$ is the identity operator in their respective spaces. It yields
    \begin{align*}
        \norm{\sum_{i=1}^{L}a_i\otimes l(x_i)}_{B(\mathcal{H})\otimes \A}\le \norm{\sum_{i=1}^{L}a_i^*a_i}_{B(\mathcal{H})}^{1/2}\norm{\sum_{i=1}^{L}l(x_i)l^*(x_i)}_{\A}^{1/2},
    \end{align*}
    Similarly, for $c_i=a_i\otimes 1$ and $d_i=1\otimes l^*(x_i)$, we have
    \begin{align*}
        \norm{\sum_{i=1}^{L}a_i\otimes l^*(x_i)}_{{B(\mathcal{H})\otimes \A}}\le \norm{\sum_{i=1}^{L}a_ia_i^*}^{1/2}_{B(\mathcal{H})}\norm{\sum_{i=1}^{L}l(x_i)l^*(x_i)}_{\A}^{1/2}.
    \end{align*}
    Combining both estimates, we obtain 
    \begin{align}\label{equation: ov equation}
        \norm{T}_{B(\mathcal{H})\otimes \A}\le 2|T|_2\norm{\sum_{i=1}^{L}l(x_i)l^*(x_i)}_{\A}^{1/2},
    \end{align}
    where we define
    \begin{align*}
        |T|_2:=\max\left\{\norm{\sum_{i=1}^{L}a_ia_i^*}_{B(\mathcal{H})}^{1/2},\norm{\sum_{i=1}^{L}a_i^*a_i}_{B(\mathcal{H})}^{1/2}\right\}.
    \end{align*}
    Let $\mathcal{L}=\sum_{i=1}^{L}l(x_i)l^*(x_i)$.
    Note that $\mathcal{L}\succeq 0$, i.e., $\mathcal{L}=\mathcal{L}^*$ and $\langle x,\mathcal{L}x\rangle \ge 0$ for all $x\in \Fock(G)$. To compute $\norm{\mathcal{L}}$, let $y\in \Fock(G)$. We denote
    \begin{align*}
        y=\sum_{w\in \T(G)}\beta_w x_w,\quad \text{where }\beta_w=\langle y,x_w \rangle.
    \end{align*}
    We write 
    \begin{align*}
        \langle y,\mathcal{L}y\rangle =\sum_{w_1,w_2\in \T(G)}\beta_{w_1}\overline{\beta}_{w_2}\sum_{i=1}^{L}\langle x_{w_1}, l(x_i)l^*(x_i) x_{w_2}\rangle.
    \end{align*}
    Note that $l(x_i)l^*(x_i)x_w=x_w$ whenever $l^*(x_i)x_w\neq 0$. Thus
    \begin{align*}
        \langle y,\mathcal{L}y\rangle 
        &=\sum_{w\in \T(G)}|\beta_w|^2|\{i\in [L]: l^*(x_i)x_w\neq 0\}|.
    \end{align*}
    Fix $w\in \T(G)$, and let $w=c_1\cdots c_h$ be its  Cartier-Foata decomposition into a unique normal sequence of cliques. Write $c_1=t_1\cdots t_k$, where $(t_{\ell},t_{\ell'})\in E(G)$ for all $\ell\ne \ell'$. Then $l^*(x_{i})x_w\neq 0$ if and only if $i\in \{t_1, \ldots, t_k\}$.
    It follows that
    \begin{align*}
        \langle y,\mathcal{L}y\rangle = \sum_{w\in \T(G)}|\beta_w|^2 k(w),
    \end{align*}
    where $k(w)$ denotes the size of the first clique in the Cartier-Foata decomposition of $w$. Since $k(w)\le \omega(G)$ for any $w\in \T(G)$, we conclude
    \begin{align*}
        \langle y,\mathcal{L}y\rangle \le \omega(G)\norm{y}_2^2,
    \end{align*}
    for any $y\in \T(G)$. Hence $\norm{\mathcal{L}}\le \omega(G)$ and the result follows by \eqref{equation: ov equation}.
\end{proof}
\begin{remark}[Left-representation of right-angled Artin groups]\label{rem: leftrep}
    Theorem \ref{thm: OperatorvaluedCoefficients} can be equally stated for the left-representation of right-angled Artin groups $\G(G)$, that is, the graph product of $\Z$ with respect to the graph $G$. We refer the reader to~\cite{green1990graph} for an introduction to the graph product of groups. Let $\lambda: \G(G) \mapsto B(l^2(\G(G)))$ be the left-representation of $\G(G)$ on $l^2(\G(G))$, then, for any Hilbert space $\mathcal{H}$ and operators $(a_i)_{i\in S}\subset B(\mathcal{H})$, where
    \begin{align*}
        S=[L]\sqcup \{1^*,\ldots, L^*\},
    \end{align*}
    we have
    \begin{align}\label{equation: left-representation}
        \norm{\sum_{i\in S}a_i\otimes \lambda(i)}_{B(\mathcal{H}) \otimes C_{\lambda}^*(\G(G))} \le 2\sqrt{\omega(G)}\max \left\{\norm{\sum_{i\in S}a_ia_i^*}^{1/2}_{B(\mathcal{H})},\norm{\sum_{i\in S}a_i^*a_i}^{1/2}_{B(\mathcal{H})}\right\}.
    \end{align}
    Here, for $i\in [L]$, we define $\lambda(i^*)=\lambda(i^{-1})$. The proof of \eqref{equation: left-representation} follows an argument analogous to that of \cite[Proposition 1.1]{haagerup1993}. As in Theorem \ref{thm: OperatorvaluedCoefficients}, \eqref{equation: left-representation} relies on a normal form for words $w\in \G(G)$ reminiscent of the Cartier-Foata decomposition for trace monoids (see \cite[Theorem 3.9]{green1990graph} and \cite[Section 2.3]{Charney2007}).

\end{remark}

\section{Further discussions}\label{section: further discussions}
In this section, we explore different approaches that yield suboptimal bounds, yet provide connections with various techniques that could potentially be refined and recycled in future extensions. Recall that
\begin{align*}
    T_{G}=\frac{1}{\sqrt{L}}\sum_{i=1}^L s_i.
\end{align*}
Starting from Lemma~\ref{lem: moments of convolution in terms of homomorphisms}, we have the identity
\begin{align*}
    \tau (T_{G}^{2p})=\frac{1}{L^p}\sum_{\pi\in P_2(2p)} \hom(F_\pi,G).
\end{align*}
One approach is to implement a compression or comparison principle, inspired by techniques from random matrix theory \cite{Bandeira_2016},  to compare $\norm{T_{G}}$ with $T_{K_{\chi}}$, where $\chi=\chi(G)$ is the chromatic number of $G$ and $K_{\chi}$ is the complete graph on $\chi$ vertices. The key observation is that  $\hom(F_\pi,G)>0$ implies that $F_\pi$ is $\chi(G)$-colorable. Since $T_{K_{\chi}}$ corresponds to a classical convolution of semicircles, this yields the bound $\norm{T_{G}} \le 2\chi(G)$. While suboptimal, it would be interesting to investigate whether such comparison arguments can be refined to yield sharp bounds.

Another natural bound arises from the constraint $\omega(F_\pi) \le \omega(G)$ whenever $\hom(F_\pi, G)>0$, leading to 
\begin{align*}
    \tau (T_{G}^{2p})\le |P_2^{(\omega(G))}(2p)|,
\end{align*}
where $P_2^{(\omega(G))}(2p)$ denotes the set of $(\omega(G)+1)$-noncrossing pair partitions.  The enumeration of such objects is known \cite{AsymptoticJinReidysWang}, resulting in the bound $\norm{T_{G}} \le 2 \omega(G)$. That this falls short of the sharp bound derived in this work suggests a finer classification of pair partitions: many pairings in $P_2^{(\omega(G))}(2p)$ contribute negligibly to $ \tau (T_{G}^{2p})$, i.e. they satisfy $\hom(F_\pi, G)\ll L^p$. Understanding how the structure of the partition influences $\hom(F_\pi, G)$ is a promising direction for further study. 

We also note that the bound $ \norm{T_{G}} \le 2\sqrt{\chi(G)}$ follows from the monotonicity of the norm with respect to edge addition (see Remark~\ref{remark: monotonicity}). Since $G$ is $\chi(G)$-colorable, there exists a partition $I_1,\ldots, I_{\chi(G)}$ of $[L]$ such that $E(G)\subseteq E(K_{I_1,\ldots,I_{\chi(G)}})$, where $K_{I_1,\ldots, I_d}$ denotes the complete $d$-partite graph over $I_1,\ldots, I_d$. The corresponding norm satisfies
\begin{align*}
    \norm{T_{K_{I_1,\ldots, I_{\chi(G)}}}}=2\sum_{k=1}^{\chi(G)}|I_k|^{1/2},
\end{align*}
from which the bound follows via the Cauchy–Schwarz inequality. 

As mentioned in the introduction, it would be desirable to develop a Turán-type approach to prove Corollary~\ref{cor: turan}. Moreover, it is unclear whether a more refined monotonicity property holds. Specifically, given two graphs  $G,G'$ on $[L]$ with the same clique number and  $|E(G)|\le |E(G')|$, does it follow that $
        \norm{T_{G}}\le \norm{T_{G'}}$?
A positive answer would suggest a constructive route to identifying extremizers and potentially establish uniqueness. More broadly, a better understanding of which graph parameters control $\norm{T_{G}}$, beyond the clique number, remains an open problem of independent interest. 

Finally, we note that Haagerup-type inequalities \cite{Haagerup1978, Jolissaint1990, Buchholz1999, Kemp2007, delaSalle2009}, which generalize Khintchine-type bounds, offer a natural direction for extending our results. Adapting such inequalities to the graph convolution framework introduced here could lead to further structural insights.




\printbibliography
\end{document}